\documentclass[a4paper,11pt,leqno]{article}
\usepackage{amsfonts,amssymb,amsmath,amsthm}
\usepackage{enumerate}
\usepackage{a4wide}
\usepackage{color}
\usepackage{eucal}

\newtheorem{theorem}{Theorem}[section]
\newtheorem{proposition}[theorem]{Proposition}
\newtheorem{corollary}[theorem]{Corollary}
\newtheorem{lemma}[theorem]{Lemma}
\theoremstyle{remark}
\theoremstyle{definition}
\newtheorem{remark}[theorem]{Remark}

\newtheorem{definition}[theorem]{Definition}

\numberwithin{equation}{section}

\newcommand{\eqdef}{\stackrel{{\rm{def}}}{=}}
\newcommand{\N}{{\mathbb N}}
\newcommand{\R}{{\mathbb R}}

\newcommand{\beq}{\begin{equation}}
\newcommand{\eeq}{\end{equation}}

\newcommand{\be}{\begin{equation}}
\newcommand{\ee}{\end{equation}}
\newcommand{\bea}{\begin{eqnarray}}
\newcommand{\eea}{\end{eqnarray}}

\begin{document}
\title{Biharmonic equation  with Singular nonlinearity in $\R^N$
\thefootnote\relax\footnote{
All the authors of this work were supported by IFCAM.}
}

\author{
{ \sc Jacques Giacomoni and Guillaume Warnault} \\
{\sc\footnotesize LMAP (UMR CNRS 5142) Bat. IPRA,}\\
{\sc\footnotesize   Avenue de l'Universit\'e }\\
{\sc\footnotesize   F-64013 Pau, France}
\\
{\sc\footnotesize  e-mail: jacques.giacomoni@univ-pau.fr}\\
 \and 
 {\sc              S. Prashanth} \\
 {\sc\footnotesize TIFR-Centre For Applicable Mathematics}\\
 {\sc\footnotesize Post Bag No. 6503, Sharada Nagar,}\\
 {\sc \footnotesize GKVK Post Office,}\\
 {\sc\footnotesize  Bangalore 560065, India} \\
 {\sc\footnotesize e-mail: pras@math.tifrbng.res.in}\\
}

\date{\today}

\maketitle

\section{Introduction}
Let $\Omega$ be a bounded  smooth domain in $\R^N$, $N\geq 2$ and denote $\rho(x):=d(x,\partial\Omega),\  x \in \Omega$. Denote by $\phi_1$ and $\lambda_1$ respectively the first (positive) eigenfunction and the first eigenvalue of $-\Delta$ in the space $H^1_0(\Omega)$.  Also, let $G$ denote the positive Green's function for $-\Delta$ in $\Omega$. Assume that  $K\in C^\nu_{loc}(\Omega)$ ($\nu\in(0,1)$) is  such that
$$
 \inf_{\Omega} K >0, \;\;\text{ and }\;\; K =O(\rho^{-\beta}) \;\;\text{ near }\;\; \partial \Omega, \;\;\text{ for some }\;\; \beta \geq 0.
$$
Given $\alpha>0$, we consider the following fourth order singular elliptic  problem:
\begin{eqnarray*}
( P)\qquad \displaystyle\left\{\begin{array}
{ll}
 & \Delta^2 u
 =  K(x)u^{-\alpha}
  \quad \mbox{ in }\,\Omega , \\
&u> 0\quad \mbox{ in }\,\Omega, \;\;u\vert_{\partial\Omega}=0, \,\Delta u\vert_{\partial\Omega} = 0.
\end{array}\right.
 \end{eqnarray*} 

There is a large literature concerning such singular problems (as well as the corresponding systems) for  second order elliptic operators wherein   questions of existence, uniqueness and multiplicity, regularity, asymptotic behaviour, symmetry, etc. have been investigated (see for instance   \cite{AdGi}, \cite{BaGi}, \cite{CaGi}, \cite{CrRaTa}, \cite{DhGiPrSa}, \cite{Gh1}, \cite{Ha},  \cite{HeMaVe}, \cite{HiSaSh}). Similar results for the  quasilinear  case have been obtained in  \cite{GiScTa} and \cite{GiScTa1}. We refer  the reader to the two excellent surveys  \cite{GhRa} and \cite{HeMa} for more details. 

There are very few results available which concern fourth order singular problems similar to $(P)$. In \cite{Gh}, the author studies the problem $ \Delta^2 u = u^{-\alpha}$, $\alpha<1,$ but with  Dirichlet boundary condition. Furthermore, the author assumes that the domain is a perturbation of the ball to ensure  positivity of the associated Green's function. Using the Schauder  fixed point theorem to a suitable integral formulation of the problem  in an appropriate cone of positive continuous functions,  the existence and the uniqueness of a solution in $C^2(\Omega)\cap C^1_0(\overline{\Omega})$ that behaves like $\rho^2$ near the boundary is shown in this work. Since such a boundary  behaviour is expected, the restriction $\alpha<1$ is  necessary. 

In contrast with \cite{Gh}, we consider the problem $(P)$ for a general smooth bounded domain $\Omega$ with  Navier boundary conditions.  We first clarify the notion of a solution to $(P)$:
\begin{definition}\label{first-defi}
A function $u\in C^2(\overline{\Omega})$ is a solution to $(P)$ if 
$u>0$ in $\Omega$, $u=\Delta u=0$ on $\partial\Omega$ and satisfies the following integral identity for any $\psi\in C^2(\overline{\Omega})\cap C_0(\overline{\Omega})$:
\begin{equation}\label{varia-form}
\int_{\Omega}\Delta u\Delta\psi{\rm d}x=\int_{\Omega}K(x) u^{-\alpha}\psi{\rm d}x.
\end{equation}
\end{definition}
\begin{remark}\label{remark1}
\begin{enumerate}
\item We require $C^2(\overline{\Omega})$ regularity to be able to define $\Delta u=0$ on $\partial\Omega$.
\item A consequence of the above definition is that  a solution $u$ to $(P)$ necessarily satisfies 
$$\int_{\Omega} K(x) u^{-\alpha}(x)\rho(x){\rm d}x<\infty.$$ To see this, plug in  the test function $\psi=\phi_1$ in \eqref{varia-form}, where $\phi_1$ is the first (normalized) positive eigenfunction  of $-\Delta$  on  $H^1_0(\Omega)$.
\item Definition \ref{first-defi} is similar to the concept of very weak solution given in \cite{BrMaPo} (see definition 0.2 there) for solving second order elliptic problems with $L^1$- ( or measure) data. We adapt this notion here for fourth order elliptic equations.
\end{enumerate}
\end{remark}
The  solution to  $(P)$ can be defined equivalently  using the  Green's representation formula (see proposition \ref{defi-equiv} in section \ref{section3}).
It is easy to see that the equation in $(P)$ is equivalent to the following second order elliptic system:
\begin{eqnarray*}
( PS)\qquad \displaystyle\left\{\begin{array}
{ll}
 & -\Delta u
 = v 
  \quad \mbox{ in }\,\Omega , \, u> 0\quad \mbox{ in }\,\Omega,\\
&-\Delta v=K(x) u^{-\alpha} \quad \mbox{ in }\,\Omega ,\\
& \;\;u\vert_{\partial\Omega}=0, \,v\vert_{\partial\Omega} = 0.
\end{array}\right.
 \end{eqnarray*} 
Nevertheless,  $(PS)$ is not a cooperative system and hence monotone methods can not be used to prove existence of solutions to $(PS)$, as  is done in \cite{CrRaTa} for the single equation. Furthermore, for $\alpha\in (0,1)$, the problem $(P)$ has a variational structure and the energy functional $J$ associated to $(P)$ is defined as follows:  
\begin{eqnarray}\label{china}
\qquad J(w)\eqdef\frac{1}{2}\int_{\Omega}(\Delta w)^2 {\rm d}x- \frac{1}{1-\alpha}\int_{\Omega} K(x)w^{1-\alpha}{\rm d}x \;\;\text{ for }\;\; w\in X\eqdef H^2(\Omega)\cap H^1_0(\Omega).
\end{eqnarray}
Clearly, $J$ is well defined in the cone of nonnegative functions in $X$ provided $K$ has moderate singularity near $\partial\Omega$. But the main difficulty is that truncation  techniques (which work in case of second order elliptic equations) can not be used directly since we are in the $H^2$-framework. This makes it difficult to employ variational methods for studying $(P)$. Another difficulty is that the Schauder fixed point theorem (used in \cite{Gh}) works only in the case $\alpha<1$ where the invariance of the solution operator with respect to a cone of positive solutions can be ensured.

For these reasons,  our approach in this paper is slightly different. We first approximate the singular problem $(P)$ by a family of problems $(P_\epsilon)$ with regular terms  as given below and use apriori estimates to show the existence of solution.
We now state the results that we prove:

\begin{theorem}\label{main}
Assume that $\alpha+\beta<2$. Then there exists a unique solution $u$ to $(P)$. Furthermore, there exist $c_1, c_2>0$ such that
\begin{eqnarray}\label{behaviour-bound}
c_1 \rho(x)\leq u(x)\leq c_2 \rho(x).
\end{eqnarray}
\end{theorem}

The idea behind the proof (see section \ref{section2}) is  to approximate the  problem in the following way:
\begin{eqnarray*}
( P_\epsilon)\qquad \displaystyle\left\{\begin{array}
{ll}
 & \Delta^2 u
 =  K_\epsilon(x) (u+\epsilon)^{-\alpha}
  \quad \mbox{ in }\,\Omega , \;\;  K_\epsilon:=\min(\frac{1}{\epsilon}, K),\\
&u> 0\quad \mbox{ in }\,\Omega, \;\;u\vert_{\partial\Omega}=0,\, \Delta u\vert_{\partial\Omega} = 0.
\end{array}\right.
 \end{eqnarray*} 
 The existence of the solution $u_\epsilon$  to $(P_\epsilon)$ can be obtained by the Schauder fix point theorem. We then prove a priori estimates on $\{u_\epsilon\}_{\epsilon>0}$ using crucially the restrictions on $\alpha,\beta$ and pass to the limit as $\epsilon\to 0^+$. 
 
 %For that we use the following {\color{magenta} Hopf principle} proved in \cite{BrCa} (see lemma 3.2 page 12):\\
%Let  $d(x)\eqdef {\rm d}(x,\partial\Omega)$, $0\leq f\in L^\infty(\Omega)$. Then, the unique solution $u$ to
%\begin{eqnarray*}
%\displaystyle\left\{\begin{array}
%{ll}
% & -\Delta u
% =  f
%  \quad \mbox{ in }\,\Omega ,\\
%&u\vert_{\partial\Omega}=0
%\end{array}\right.
% \end{eqnarray*} 
%satisfies for $x\in\Omega$
%\begin{equation}\label{1}
%u(x)\geq Cd(x)\int_{\Omega}f(y)d(y){\rm d}y
%\end{equation}
%where $C>0$ is independent of $f$. Applying twice the inequality \eqref{1}, we obtain that \eqref{1} holds for the biharmonic equation. Combined \eqref{1} with a priori estimates, we prove the following main result:

The following nonexistence result proves that the restriction $\alpha+\beta<2$ is sharp in the above results:
\begin{theorem}\label{nonexistence}
Assume that $\alpha+\beta\geq 2$. Then, there is no solution to $(P)$.
\end{theorem}

Next, we use Theorem \ref{main} to obtain the existence of a path-connected branch of solutions to the following bifurcation problem: 
\begin{eqnarray}\label{bifurcation-problem}
( P_\lambda)\qquad \displaystyle\left\{\begin{array}
{ll}
 & \Delta^2 u
 =  K(x) u^{-\alpha}+\lambda f(u)
  \quad \mbox{ in }\,\Omega ,\\
&u> 0\quad \mbox{ in }\,\Omega, \;\; u\vert_{\partial\Omega}=0, \,\Delta u\vert_{\partial\Omega} = 0
\end{array} \right.
\end{eqnarray} 
where $\lambda$ is the  bifurcation parameter and $f$ a  function  satisfying the following assumptions:
\begin{itemize}
\item[($f_0$)] $f:[0,\infty) \to [0,\infty)$ is a twice continuously differentiable map with $f(0)=0$. 
\item[($f_1$)]  $f(t)$  is a finite product of functions of the form  $g(t^{p}), p > 0,$ where $g$  is a  real entire function on $\R$. 
\item[($f_2$)]  $f' \geq 0$ and \; $\displaystyle \liminf_{t \to \infty} \frac{f(t)}{t} >0.$ 
\end{itemize}

Given a positive continuous function $\phi$ on $\Omega$, denote by 
$$\mathcal{C}_{\phi}(\Omega):= \Big \{ u\in C(\overline{\Omega}): \displaystyle \sup_{\Omega} \Big|\frac{u}{\phi} \Big| < +\infty \Big \},$$
with the norm
$$
\| u\|_{\mathcal{C}_{\phi}(\Omega)}:=  \displaystyle \sup_{\Omega} \Big|\frac{u}{\phi} \Big| 
$$
and the ``positive cone"
$$\mathcal{C}_{\phi}^+(\Omega):= \Big \{ u\in \mathcal{C}_{\phi}(\Omega): \displaystyle \inf_{\Omega} \frac{u}{\phi} >0 \Big \}.$$  
We define the inverse of the biharmonic operator denoted as $(\Delta^2)^{-1}$ as follows:
$$
(\Delta^2)^{-1}h = u 
$$
where for $h$ in an appropriate space, $u$ solves the inhomogeneous problem:
\begin{eqnarray}\label{biharmonic-f}
\displaystyle\left\{\begin{array}
{ll}
 & \Delta^2 u
 =  h
  \quad \mbox{ in }\,\Omega ,\\
&u\vert_{\partial\Omega}= \Delta u\vert_{\partial\Omega} = 0.
\end{array}\right.
\end{eqnarray}

The bifurcation analysis is done in the space $\R \times \mathcal{C}_{\phi_1}(\Omega)$. Therefore, we consider the following set of all  solutions (in the sense of definition \ref{first-defi})
\begin{equation}\label{soln set}
\mathcal{S}= \{u\in C^4(\Omega)\cap C^2(\overline{\Omega}), u>0 \text{ solves } (P_\lambda)\} \subset \R \times \mathcal{C}_{\phi_1}(\Omega).
\end{equation}
Consider the following solution operator associated to $(P_\lambda)$ : 
\begin{equation}
F(\lambda ,u)=u- (\Delta^2 )^{-1}(K(x)u^{-\alpha }+\lambda f(u)),\;  
(\lambda ,u)\in \mathbb{R} \times {\mathcal C}_{\phi_1 }^+(\Omega ), \;0<\alpha+\beta<2.  \label{3}
\end{equation}
 Using the framework of analytic bifurcation theory as developed in the works \cite{BuDaTo} and \cite{BuDaTo2} (see also  \cite{BoGiPr} and \cite{BuTo}), we obtain an analytic global unbounded path of solutions to $(P_\lambda)$:
\begin{theorem}\label{th1}
Let $f$ satisfy conditions $(f_0)-(f_2)$ and $\alpha+\beta<2$.
	Then, $F : \mathbb{R} \times {\mathcal C}_{\phi_1 }^+(\Omega ) \to {\mathcal C}_{\phi_1 }(\Omega )$ is an analytic map (see definition \ref{anal map}). Furthermore, there exists $\Lambda \in(0,\infty)$ and an unbounded set  ${\mathcal A} \subset (-\infty,\Lambda]\times {\mathcal C}^+_{\phi_1}(\Omega) \subset \mathcal{S}$ of solutions to $(P_\lambda)$ which is globally parametrised by a continuous  map :
	
	 $$(-\infty,\infty) \ni s \to (\lambda(s),u(s)) \in \mathcal{A} .$$
 
 Moreover, the following properties hold along this  path $\mathcal{A}$:    
\begin{itemize}
\item[(i)] $(\lambda(s),u(s))\to (0,u_0)$  in  $\R \times {\mathcal C}_{\phi_1}(\Omega)$ as $s \to 0$, where $u_0$ is the unique solution to $(P)$.
\item[(ii)] $\Vert u(s)\Vert_{{\mathcal C}_{\phi_1}(\Omega)}\to \infty$ as $s\to\infty$ . 
\item[(iii)] $\mathcal{A}$ has at least one asymptotic bifurcation point $\Lambda_a\in [0,\Lambda]$. That is, there exist sequences $\{s_n\}_{n\in\N}\subset (0,\infty)$, $\{(\lambda(s_n), u(s_n))\} \subset {\mathcal A}$ such that $s_n\to\infty$, $\lambda(s_n)\to\Lambda_a$ and $\Vert u(s_n)\Vert_{{\mathcal C}_{\phi_1}(\Omega)} \to \infty$.
\item[(iv)]   $\left\{s\geq 0\,:\,\partial_u F(\lambda(s),u(s)) \text{ is not invertible }\right\}$ is a discrete set.
\item[(v)]($\mathcal{A}$ is an ``analytic" path) At each of its points ${\mathcal A}$ has a local analytic re-parameterization in the following sense: For each $s^*\in \R$ there exists a continuous, injective map $\rho^*\,:\, (-1,1)\to \R$ such that $\rho^*(0)=s^*$  and the re-parametrisation
\begin{eqnarray*}
 (-1,1) \ni t\to (\lambda(\rho^*(t)),u(\rho^*(t))) \in \mathcal{A} \mbox{ is analytic}.
\end{eqnarray*}
Furthermore, the map $s \mapsto \lambda(s)$ is injective in a neighborhood of $s=0$ and for each $s^* >0$ there exists $\epsilon^*>0$ such that $\lambda$ is injective on $[s^*,s^*+\epsilon^*]$ and on $[s^*-\epsilon^*,s^*]$.

\item[(vi)] For any $\lambda\leq 0$, there exists atmost one solution to $(P_\lambda)$ and ${\mathcal A} \cap (-\infty,0) \times {\mathcal C}_{\phi_1}(\Omega )$ is a single analytic curve which is a graph from the $\lambda$ axis consisting of non-degenerate solutions $u_\lambda$. In particular, we can take $\lambda(s)=s$ for $s<0$.
 \end{itemize}
\end{theorem}
The paper is organized as follows: In Section \ref{section2}, we prove Theorem \ref{main} using a version of  Hopf principle recalled in proposition \ref{prop-Hopf}. In Section \ref{section3}, we study the equivalence between the two definitions of a solution and prove Theorem \ref{nonexistence}. Finally, in Section \ref{section4} we prove Theorem \ref{th1}.
\section{Some preliminary results for Theorem \ref{main}}\label{prel}
We first prove a version of Hopf principle.
\begin{proposition}\label{prop-Hopf}
Let $h\in L^\infty(\Omega)$ be a nonnegative function. Let $u$ be the classical solution  to \eqref{biharmonic-f}.
Then there exists a constant $C>0$ (independent of $h$) such that the following inequality holds:
\begin{equation}\label{1}
u(x)\geq C \rho(x)\int_{\Omega}h(y)\rho(y){\rm d}y.
\end{equation}
\end{proposition}
\proof
Since $h\in L^\infty(\Omega)$, $u$ solves the following system:
\begin{eqnarray} \label{haf}
\displaystyle\left\{\begin{array}
{ll}
 & -\Delta u
 = v 
  \quad \mbox{ in }\,\Omega ,\\
&-\Delta v=h\quad \mbox{ in }\,\Omega ,\\
& \;\;u\vert_{\partial\Omega}= v\vert_{\partial\Omega} = 0.
\end{array}\right.
 \end{eqnarray} 
Recall from  lemma 3.2 in \cite{BrCa} that 
for any nonnegative function $h\in L^\infty(\Omega)$, the unique solution $w$ to the problem
\begin{eqnarray*}
\displaystyle\left\{\begin{array}{ll}
 & -\Delta w  =  h   \quad \mbox{ in }\,\Omega ,\\
&w\vert_{\partial\Omega}=0
\end{array}\right.
\end{eqnarray*}
satisfies the estimate:
\begin{equation*}
w(x)\geq C\rho(x)\int_{\Omega}h(y)\rho (y){\rm d}y \quad x \in \Omega,
\end{equation*} 
where the constant $C$ does not depend to $h$. 

We  apply the previous inequality to $u$ and $v$  to get
\begin{eqnarray*}
u(x)\geq C \rho(x)\int_{\Omega}v(y)\rho(y){\rm d}y\geq C^2\rho(x)\int_{\Omega}\rho(y)^2{\rm d}y\int_{\Omega}h(z)\rho(z){\rm d}z
\end{eqnarray*} 
which completes the proof. \qed\\
By a simple approximation argument and the maximum principle, we have the
\begin{corollary}\label{dee}
Let $h \rho  \in L^1(\Omega)$ and nonnegative. Then any $u$ solving \eqref{biharmonic-f} (in the sense of definition \ref{first-defi}) satisfies the inequality \eqref{1}.
\end{corollary}
We next have the following regularity and uniform estimate result :
\begin{lemma} \label{jee}
Let $h\in C^\nu_{loc}(\Omega)$ be a nonnegative function such that $h \rho^{\delta} \in L^{\infty}(\Omega)$ for some $0<\delta<2$. Let  $u\in C^2(\overline{\Omega})$ be the solution (in the sense of definition \ref{first-defi})  to \eqref{biharmonic-f}. 
Then there exist constants $C>0$ (dependent on $\|h\rho^{\delta}\|_{L^{\infty}(\Omega)}$, $\nu$ and $\delta$) and  $0<\theta<1$ (depending  on $\nu$ and $\delta$)  such that the following inequality holds:
\begin{equation}\label{1b}
\|u\|_{C^{2,\theta}(\overline{\Omega})} \leq C.
\end{equation}
\end{lemma}
\begin{proof}
Since $h \rho \in L^1(\Omega)$, from the above corollary  we obtain that $u \geq c \rho$ for some $c>0$. Since $u \in C^2(\overline{\Omega}) \cap C_0(\overline{\Omega})$, we obtain that $u \sim \rho$ near $\partial\Omega$. We note that  $v:=-\Delta u\in C^{2,\nu}_{loc}(\Omega)$ by elliptic regularity and is a nonnegative function by the maximum principle. Consider the equivalent system for $u,v$ as in \eqref{haf}.  Then we have
\begin{equation} \label{cam}
\vert\Delta v\vert\leq C_0\rho^{-\delta}, \quad \text{ where }\;\;  C_0:= \|h\rho^{\delta}\|_{L^{\infty}(\Omega)}.
\end{equation}
Let $w:=w(\delta)$ denote the unique positive solution to 
\begin{eqnarray*}
\left\{\begin{array}{ll}
&-\Delta w =  w^{-\delta}\mbox{ in }\Omega,\\
 &w= 0\mbox{ on }\partial\Omega.
\end{array}\right.
\end{eqnarray*}
From \cite{CrRaTa}, there exist positive constants $c_1< c_2$  such that
 the following estimates hold:
\begin{eqnarray}
 c_1\rho \leq& w &\leq c_2\rho  \;\;\text{ if }\;\;0<\delta <1, \notag \\
 c_1\rho\ln(\frac{D}{\rho})\leq& w &\leq c_2\rho\ln(\frac{D}{\rho}) \;\;\text{ if }\;\; \delta=1 \;\; (D:= diam(\Omega)) \;\;\text{ and }\;\; \notag\\
  c_1\rho^{\frac{2}{\delta+1}}\leq& w &\leq c_2\rho^{\frac{2}{\delta+1}} \;\;\text{ if }\;\;1<\delta<2. \notag
\end{eqnarray}
Choosing the constant $M>0$ large enough (depending on $C_0,\ c_1$ and $\delta$) and using the weak comparison principle, we can conclude %$v + Mw \geq 0$ and 
$v-Mw \leq 0$. Thus, we have 
\begin{equation}\label{mil}
0\leq v \leq M w \leq Mc_2 \rho^{\mu} \;\;\text{ for some} \;\; \mu>0.
\end{equation}
%if $\delta<1$, $v\leq w_{\frac{\delta}{2-\delta}}$ if $1<\delta<2$ and $v\leq w_{1+\epsilon_0}$ for $0<\epsilon_0$ and if $\delta=1$. 
By noting \eqref{cam} and \eqref{mil}, appealing  to Proposition 3.4  in \cite{GuLi} we  obtain that 
 $v\in C^{0,\theta}(\overline{\Omega})$ for some $\theta:=\theta(\delta, \nu)\in (0,1)$ and $\Vert v\Vert_{C^{0,\theta}(\overline{\Omega})}\leq C=C(C_0, \delta)$. We then apply the classical elliptic theory to get $u\in C^{2,\theta}(\overline{\Omega})$ and $\Vert u\Vert_{C^{2,\theta}(\overline{\Omega})}\leq \tilde C=\tilde C(C_0, \delta,\nu)$.
 \end{proof}
We can now show the following result on existence of $C^2(\overline{\Omega})$ solution (as in definition \ref{first-defi}) by means of a simple approximation argument:
\begin{proposition}\label{gul}
Let $h$ be a nonnegative function such that $h \rho^{\delta} \in L^{\infty}(\Omega)$ for some $0<\delta<2$. Then there exists a unique solution $u \in C^2(\overline{\Omega})$ solving \eqref{biharmonic-f}.
\end{proposition}
\proof 
Define $h_n := \min\{h, n\}$. Let $u_n  \in C^2(\overline{\Omega})$ be the unique solution to \eqref{biharmonic-f} with $h=h_n$. We note that given $\psi\in C^2(\overline{\Omega})\cap C_0(\overline{\Omega})$ there exist $p>1$ such that
$$
\Big\{h_n \psi\Big\} \quad \text{ is a bounded sequence in } \;\; L^p(\Omega).
$$
Then, by Vitali's convergence theorem
$$
\int_{\Omega} h_n \psi \to \int_{\Omega} h \psi \;\;\text{ as }\;\; n \to \infty.
$$
By appealing to lemma \ref{jee} we obtain as well that for some $\theta \in (0,1)$,
$$
\Big\{u_n\Big\} \quad \text{ is a bounded sequence in } \;\; C^{2,\theta}(\overline{\Omega}).
$$
Thus, upto a subsequence $u_n \to u$ in $C^{2}(\overline{\Omega}).$ It is then easy to see that $u$ solves \eqref{biharmonic-f}.
\hfill \qed

We recall the Hardy Inequality for $H^s$ spaces:
\begin{lemma}[see Lemma 3.2.6.1 in \cite{Tr}]\label{Hardy} Let $s\in [0,2]$ such that $s\neq \frac{1}{2}$ and $s\neq \frac{3}{2}$. Then the following generalisation of Hardy's inequality holds:
\begin{equation}
\|\rho^{-s} g\|_{L^{2}(\Omega)}\leq C \|g\|_{H^{s}(\Omega)}\quad \mbox{ for all }g\in H_0^s(\Omega). \label{eq4b}
\end{equation}
\end{lemma}

Finally, we state the following regularity result.
\begin{lemma}\label{hum}
Assume that $h$ is a nonnegative function such that $h \rho^{\delta} \in L^{\infty}(\Omega)$ for some $0<\delta<2$. Let $u \in C^2(\overline{\Omega}) \cap C_0(\overline{\Omega})$ be the solution of \eqref{biharmonic-f}. Then $u \in W^{4,p}_{loc}(\Omega) \cap H^{4-s}(\Omega)$ for any $1 \leq p< \infty$ and $s \in (\delta -\frac{1}{2},2] \setminus \{\frac{1}{2}, \frac{3}{2} \}$.
\end{lemma}
\proof That  $u \in W^{4,p}_{loc}(\Omega)$ for $1\leq p <\infty$  follows from standard elliptic regularity result. Fix any $s \in (\delta -\frac{1}{2},2] \setminus \{\frac{1}{2}, \frac{3}{2} \}$. We claim that $ h \in H^{-s}(\Omega)$. Indeed, using lemma \ref{Hardy},  for any $\xi \in H^{s}_0(\Omega)$,
\begin{eqnarray}
\Big|\int_{\Omega }h \xi \Big | & = & \left\vert \int_{\Omega }(h \rho^{s})(\xi \rho ^{-s})\right\vert \notag \\ 
& \leq & C\int_{\Omega }(\rho^{s-\delta })(\rho^{-s}|\xi |) \notag \\ 
& \leq & C\Vert \rho^{s-\delta}\Vert
_{L^{2}(\Omega )}\Vert \rho^{-s}\xi \Vert _{L^{2}(\Omega )}
\notag \\ 
& \leq & C\Vert \xi \Vert _{H_{0}^{s}(\Omega )} \notag.
\label{17}
\end{eqnarray}
Hence by elliptic regularity used successively to $v$ and $u$ we obtain that $u \in H^{4-s}(\Omega)$. \hfill \qed

\section{Proof of Theorem \ref{main}}\label{section2}
 We first show that the solution is unique. Let $u_1$ and $u_2$ be two solutions to $(P)$. Then, 
\begin{eqnarray}\label{arg-uniqueness}
\int_{\Omega}(\Delta(u_1-u_2))^2{\rm d}x=\int_{\Omega}K(x) (u_1^{-\alpha}-u_2^{-\alpha})(u_1-u_2){\rm d}x\leq 0.
\end{eqnarray}
Therefore, since $u_1, u_2\in H^1_0(\Omega)$, we obtain $u_1\equiv u_2$.

Fix $\epsilon>0$. We next prove the existence of a unique solution to $(P_\epsilon)$. Let  ${\mathcal W}$ be the positive cone of $C_0(\overline{\Omega})$, i.e.
\begin{eqnarray*}
{\mathcal W}\eqdef \left\{u\in C_0(\overline{\Omega})\,\vert\, u\geq 0 \mbox{ in }\Omega\right\}. 
\end{eqnarray*}
 We define the functional $\Phi\, :\, {\mathcal W}\to {\mathcal W}$  as
 the solution to the following problem:
\begin{eqnarray*}
\displaystyle\left\{\begin{array}
{ll}
 & \Delta^2 \Phi(u)
 =  K_\epsilon(x)(u+\epsilon)^{-\alpha}
  \quad \mbox{ in }\,\Omega ,\\
&\Phi(u)\vert_{\partial\Omega}=0, \,\Delta \Phi(u)\vert_{\partial\Omega} = 0.
\end{array}\right.
\end{eqnarray*}
By the elliptic regularity theory, $\Phi$ is a compact  linear operator on $C_0(\overline{\Omega})$, and by the weak comparison principle, leaves the closed convex set ${\mathcal W}$ invariant. Hence, by the Schauder fixed point theorem, there exists $u_\epsilon\in {\mathcal W}$ solution to $(P_\epsilon)$. Using a similar argument as in \eqref{arg-uniqueness}, $u_\epsilon$ is the unique solution to $(P_\epsilon)$.  By elliptic regularity, we also have $u_\epsilon \in C^2(\overline{\Omega})$.

Multiplying  the equation satisfied by $u_\epsilon$ by $\phi_1$, we obtain that
\begin{equation}\label{ji}
\lambda_1^2\int_{\Omega}u_\epsilon\phi_1{\rm d}x=\int_{\Omega}\Delta^2u_\epsilon\phi_1{\rm d}x=\int_{\Omega}K_\epsilon(x)\phi_1 (u_\epsilon(x)+\epsilon)^{-\alpha}{\rm d}x. 
\end{equation}

First we show a uniform  lower bound:
\begin{proposition}\label{jyo}
There exists a constant $C>0$ independent of $\epsilon$ such that $u_\epsilon \geq C \rho$ in $ \Omega$.
\end{proposition}
\proof
We first show the following fact:
 \begin{equation}\label{cha}
 \displaystyle \inf_{\epsilon>0}\int_{\Omega} K_\epsilon \phi_1 (u_\epsilon+\epsilon)^{-\alpha}{\rm d}x>0.
 \end{equation}
  We argue by contradiction. Suppose, up to a subsequence,  
$$\int_{\Omega}K_\epsilon \phi_1 (u_\epsilon+\epsilon)^{-\alpha}{\rm d}x\to 0 \;\;\text{ as }\;\; \epsilon\to 0^+.$$

 Using \eqref{ji} this implies that $\int_{\Omega}u_\epsilon \phi_1{\rm d}x\to 0$ and hence 
$
u_\epsilon\to 0$ in $L^1_{\rm loc}(\Omega)$ as $\epsilon\to 0^+$.

 Again up to a subsequence, we deduce that 
\begin{eqnarray}\label{conv-pp}
u_\epsilon\to 0  \;\; \text{ and }\;\; K_\epsilon (u_\epsilon +\epsilon)^{-\alpha}\to 0\;\mbox{ a.e. in } \Omega\mbox{ as }\epsilon\to 0^+
\end{eqnarray}
 which is a contradiction.  This proves  \eqref{cha} above.
 
By elliptic regularity theory, $u_\epsilon\in C^{2,\gamma}(\overline{\Omega})$ for any $\gamma\in (0,1)$ and from Proposition \ref{prop-Hopf} the estimate 
\begin{eqnarray}\label{Hopf-epsilon}
u_\epsilon(x)\geq C\rho(x)\int_{\Omega}K_\epsilon (u_\epsilon+\epsilon)^{-\alpha}\rho \;{\rm d}y
\end{eqnarray}
holds. The conclusion follows from \eqref{cha}.
\hfill \qed

  \begin{proposition}\label{chr}
There exists $\theta \in (0,1)$ independent of $\epsilon>0$ such that
$$
\sup_{\epsilon>0} \|u_\epsilon\|_{C^{2,\theta}(\Omega)} < +\infty.
$$
\end{proposition}
\proof From the last proposition, it follows that 
\begin{equation}\label{ahm}
K_\epsilon (u_\epsilon + \epsilon)^{-\alpha} \rho^{\alpha+\beta} \in L^{\infty}(\Omega). 
\end{equation}
Noting that $0<\alpha+\beta<2$ and invoking lemma \ref{jee}, the conclusion follows. \hfill \qed

Let $ u_\epsilon \to u$ in $C^2(\overline{\Omega})$ as $\epsilon \to 0$.  From \eqref{ahm}, we note  that  given $\psi\in C^2(\overline{\Omega})\cap C_0(\overline{\Omega})$ there exists $p>1$ such that
$$
\Big\{K_\epsilon (u_\epsilon + \epsilon)^{-\alpha} \psi \Big\}_{\epsilon>0} \quad \text{is a bounded family in } L^{p}(\Omega).
$$
We can  now  use Vitali's convergence theorem to directly pass to the limit as $\epsilon \to 0$ in $(P_\epsilon)$ to conclude that $u$ solves $(P)$.

\section{ Proof of Theorem \ref{nonexistence}}\label{section3}
We first prove the following equivalent way of defining a solution to $(P)$:
\begin{proposition}\label{defi-equiv}
$u\in C^2(\overline{\Omega})\cap C_0(\overline{\Omega})$ is a solution to $(P)$ (in the sense of definition \ref{first-defi}) if $u>0$ in $\Omega$ and verifies
\begin{eqnarray}\label{green-funct}
u(x)=\int_{\Omega}G(x,y)\left(\int_{\Omega}G(y,z)K(z) u^{-\alpha}(z) {\rm d}z\right){\rm d}y.
\end{eqnarray}
\end{proposition}

\proof 
Assume first that $u$ satisfies Definition \ref{first-defi}. From the estimates in Proposition 4.13 in \cite{GaGrSw} and noting $\int_{\Omega}K(z)\rho(z)u^{-\alpha}(z){\rm d}z<\infty$ (see Remark \ref{remark1}), we obtain by Fubini's theorem that for any $x\in \Omega$,
\begin{eqnarray*}
\int_\Omega G(x,y)\left(\int_\Omega G(y,z)K(z) u^{-\alpha}(z) {\rm d}z\right){\rm d}y=\int_{\Omega}K(z)u^{-\alpha}(z) {\rm d}z \int_{\Omega}G(x,y)G(y,z){\rm d}y<\infty.
\end{eqnarray*}
Therefore, by classical arguments, $u$ satisfies \eqref{green-funct}.

Now assume that $u\in C^2(\overline{\Omega})\cap C_0(\overline{\Omega})$, $u>0$ in $\Omega$ and verifies \eqref{green-funct}. Let us show that $u$ satisfies Definition \ref{first-defi}. For that, observe that for $\eta>0$ small enough,
\begin{eqnarray}\label{phi1-control}
u(x)-\frac{\eta}{\lambda_1^2}\phi_1(x)=\int_{\Omega}G(x,y)\left(\int_{\Omega}G(y,z)(K(z) u^{-\alpha}(z)-\eta\phi_1(z)){\rm d}z\right)\geq 0.
\end{eqnarray}
Thus $u \geq c \rho$ for some $c>0$. From the $C^2$-regularity of $u$, we also have that $u \leq C\rho$ for $C>0$.  Therefore, by the assumptions on $K$, for any $\psi\in C^2(\overline{\Omega})\cap C_0(\overline{\Omega})$,
\begin{eqnarray*}
\int_{\Omega} K(x) u^{-\alpha}(x) \psi(x){\rm d}x<\infty
\end{eqnarray*}
and hence $u$ satisfies \eqref{varia-form}.\qed

Now we prove Theorem \ref{nonexistence}.
\proof
Let $\alpha+\beta\geq 2$ and $u$ be a solution to $(P)$. From Proposition \ref{defi-equiv}, the inequality  \eqref{phi1-control} holds and noting that $u\in C^1(\overline{\Omega})$, we obtain that  $u \sim \rho$ in $\Omega$. Since $\alpha+\beta\geq 2$, from Theorem 2.4 in \cite{Gh1}, we get the required contradiction.\qed
\section{Bifurcation results}\label{section4}
In this section we prove Theorem \ref{th1}. We consider the following bifurcation framework (see Chapter 9 in \cite{BuTo} or Theorem 1.13 in \cite{BoGiPr} for more details):

Let $\mathcal{X}$, $\mathcal{Y}$ be real Banach spaces, $\mathcal{U}\subset\R^+\times \mathcal{X}$ an open set. Let $\Psi : \mathcal{U} \to \mathcal{Y}$ be a map.
\begin{definition}\label{anal map}
$\Psi$  is said to be 
real analytic on $\mathcal{U}$ if for each $x \in \mathcal{U}$ there is an $\varepsilon > 0$ and continuous $k$-homogeneous
polynomials $P_k : \mathcal{U} \to \mathcal{Y}$  such that
$\Psi(x + h) = \sum_{k=0}^{\infty} P_k (h)$ if $\|h\| < \varepsilon$.
\end{definition}

 Define the solution set $${\mathcal S}=\displaystyle\left\{(\lambda,x)\in \mathcal{U}\,:\, \Psi(\lambda,x)=0\right\}$$
 and the non-singular solution set
$$\mathcal{N}= \displaystyle\left\{(\lambda,x)\in \mathcal{S}\,:\,Ker( \partial_x \Psi(\lambda,x))=\{0\}\right\}.$$
\begin{definition}\label{def dist arc}
A distinguished arc is a maximal connected subset of $\mathcal{N}$.
\end{definition}
 Suppose that
\begin{itemize}
\item[(G1)] Bounded closed subsets of ${\mathcal S}$ are compact in $\R \times \mathcal{X}$.
\item[(G2)]  $\partial_x\Psi(\lambda, x)$ is a Fredholm operator of index zero for all $(\lambda,x) \in \mathcal{S}$. 
\item[(G3)] There exists an analytic function $(\lambda,u) \, : (-\epsilon,\epsilon) \to \mathcal{S}$ such that     $\partial_x \Psi(\lambda(s),u(s))$ is invertible for all $s\in (-\epsilon,\epsilon)$ and  $\displaystyle\lim_{s\to 0^+}(\lambda(s), u(s))=(0,u_0)$ where $u_0\in {\mathcal X}$ is the unique solution to $\Psi(0,u_0)=0$.
\end{itemize}
 Let
$$ {\mathcal A_0}=\left\{(\lambda(s),u(s))\,:\, s\in (-\epsilon,\epsilon)\right\}.$$
Obviously, ${\mathcal A_0}\subset {\mathcal S}$. The following result gives a global extension of the function $(\lambda, u)$ from $(-\epsilon,\epsilon)$ to $(-\infty,\infty)$ in the real analytic case.
\begin{theorem}\label{theo9.1.1}
Suppose (G1)-(G3) hold. Then,  $(\lambda,u)$ can be extended as a continuous map (still called) $(\lambda,u) :  (-\infty,\infty) \to  \mathcal{S}$ with the following properties:
\begin{itemize}
\item[(a)] Let ${\mathcal A} \eqdef \{(\lambda(s),u(s)): s \in\R\}.$ Then, ${\mathcal A} \cap {\mathcal N}$ is an atmost countable union of distinct distinguished arcs $\bigcup_{i=1}^n {\mathcal A_i},\; n \leq \infty$.
\item[(b)] ${\mathcal A_0}\subset {\mathcal A_1}$.
\item[(c)]  $\{s\in\R\,:\,ker(\partial_x \Psi(\lambda(s),u(s))) \neq \{0\}\}$ is a discrete set.
\item[(d)] At each of its points ${\mathcal A}$ has a local analytic re-parameterization in the following sense: For each $s^*\in \R$ there exists a continuous, injective map $\rho^*\,:\, (-1,1)\to \R$ such that $\rho^*(0)=s^*$  and the re-parametrisation
\begin{eqnarray*}
 (-1,1) \ni t\to (\lambda(\rho^*(t)),u(\rho^*(t))) \in \mathcal{A} \mbox{ is analytic}.
\end{eqnarray*}
Furthermore, the map $s \mapsto \lambda(s)$ is injective in a  neighborhood of $s=0$ and for each $s^*\neq 0$ there exists $\epsilon^*>0$ such that $\lambda$ is injective on $[s^*,s^*+\epsilon^*]$ and on $[s^*-\epsilon^*,s^*]$.
\item[(e)] Only one of the following alternatives occurs:
\begin{itemize}
\item[(i)] $\Vert(\lambda(s),u(s))\Vert_{\R \times \mathcal{X}}\to\infty$ as $s\to +\infty$ (resp. $s \to -\infty$).
\item[(ii)] a subsequence $\{(\lambda(s_n),u(s_n))\}$ approaches the boundary of $\mathcal{U}$ as $s_n \to +\infty$ (resp. $s_n \to -\infty$).
\item[(iii)] ${\mathcal A}$ is the closed loop :
$$
{\mathcal A}=\left\{(\lambda(s),u(s))\,:\, -T\leq s\leq T, (\lambda(T),u(T))=(\lambda(-T),u(-T)) \text{ for some } T>0 \right\}.
$$ In this case, choosing the smallest such $T>0$ we have
\begin{eqnarray*}
(\lambda(s+2T),u(s+2T))=(\lambda(s),u(s)) \mbox{ for all } s \in \R.
\end{eqnarray*}
\end{itemize}
\item[(f)] Suppose $\partial_x \Psi(\lambda(s_1),u(s_1))$  is invertible  for some $s_1\in\R$. If for some $s_2\neq s_1$, we have 
$(\lambda(s_1),u(s_1))=(\lambda(s_2),u(s_2))$
then (e)(iii) occcurs and $\vert s_1-s_2\vert$ is an integer multiple of $2T$. In particular, the map $s \mapsto (\lambda(s), u(s))$ is  injective on $[-T,T)$.
\end{itemize}
\end{theorem}
\begin{remark}\label{bif from u0}
We remark that theorem 9.1.1 in \cite{BuTo} deals with  ``bifurcation from the first eigenvalue" type of situation whereas Theorem 1.13 in \cite{BoGiPr} concerns the bifurcation from origin.
The conditions $(G1)-(G3)$  assumed there are required only to ensure that the starting analytic path corresponding to $\mathcal{A}_0$ is available for global extension. In our case, we make this as an assumption $(G3)$ above. Hence the proof given in \cite{BuTo} and in \cite{BoGiPr} holds good in our case as well.
\end{remark}
We recall the following result from \cite{BoGiPr} (proposition 2.1). \begin{proposition} \label{g anal}
Let $g : \R \to \R$ be an entire function with $g(0)=0$. Define $M_k (a)= max_{[-a,a]}g^{(k)}, \; k=1,2,3,..$. Assume that for any $a \geq 0$, there exists $\mu>0$ such that the series $\sum_{k=0}^{\infty} \frac{M_k(a)}{k!}\mu^k $ converges. Then, for any $\phi \in  C_0(\overline{\Omega}), \phi>0$ in $\Omega$, we have ${\mathcal C}_\phi(\Omega) \ni u \mapsto g(u) \in {\mathcal C}_\phi(\Omega)$ is an analytic map. Furthermore, if  $\inf_{[0,\infty)}g' >0$, then $g$ maps $
{\mathcal C}_\phi^+(\Omega)$ into itself.
\end{proposition}
Consider now the  solution operator $F$ associated to $(P_\lambda)$ defined in \eqref{3}.

\begin{proposition}\label{analyticity}   
The map $F$ takes $\mathbb{R}\times {\mathcal C}_{\phi_1 }^+(
\Omega )$  into ${\mathcal C}_{\phi_1}(\Omega)$ and is analytic. 
Furthermore, if $\lambda\geq 0$, then $F(\lambda, \cdot)$ maps $ {\mathcal C}_{\phi_1 }^+(
\Omega )$ into ${\mathcal C}_{\phi_1}^+(\Omega)$.
\end{proposition}

\proof

{\it Step 1: The map ${\mathcal C}_{\phi_1 }^+(\Omega )\ni u\longmapsto
K(x)u^{-\alpha } + \lambda f(u) \in {\mathcal C}_{\phi_1 ^{-\alpha -\beta }}(\Omega )$ is
analytic.} 

Given $u\in {\mathcal C}_{\phi_1 }^+(\Omega )$, it follows that $K(x)u^{-\alpha} \in {\mathcal C}_{\phi_1^{-\alpha-\beta }}(\Omega )$. Then following the arguments in Step 1 of prop.2.3 in  \cite{BoGiPr}, we obtain 
 the analyticity of the map. 

{\it Step 2: The map ${\mathcal C}_{\phi_1 ^{-\alpha-\beta }}(\Omega )\ni u\longmapsto (\Delta^2)
^{-1}u\in {\mathcal C}_{\phi_1}(\Omega)$ is a linear continuous map (and hence
analytic). Furthermore, this map takes ${\mathcal C}^+_{\phi_1 ^{-\alpha-\beta }}(\Omega )$ into ${\mathcal C}^+_{\phi_1}(\Omega)$.}

We observe that $(\Delta^2)^{-1}$ is well defined on ${\mathcal C}_{\phi_1 ^{-\alpha-\beta }}(\Omega )$.  Indeed, since $\alpha+\beta<2$,  from  lemma \ref{jee}, there exists a unique solution $w\in C^{2, \theta}(\overline{\Omega }), 0<\theta<1,$ solving
\begin{equation}
\left\{ 
\begin{array}{ll}
\Delta^2 w=u\ \ in\ \ \Omega ,\ \ u\in {\mathcal C}_{\phi_1 ^{-\alpha-\beta }}(\Omega ), &  \\ 
w=\Delta w=0 \mbox{  on} \ \partial \Omega. & 
\end{array}%
\right.  \label{19}
\end{equation}%
Clearly, $w := (\Delta^2)
^{-1}u \in {\mathcal C}_{\phi_1 }(\Omega )$.
If additionally $u \in {\mathcal C}^+_{\phi_1 }(\Omega )$, from the  Hopf principle  in corollary \ref{dee}, we also have that $w\in {\mathcal C}_{\phi_1}^+(\Omega)$. The proof of the proposition follows by combining steps 1 and 2.\hfill\qed \newline
We now prove the existence of ${\mathcal A_0}$:
\begin{proposition}\label{A+}
Let $0<\alpha+\beta<2$. There exists a $\lambda_0>0$ such that  for all $\lambda\in (-\lambda_0,\lambda_0)$, there exists a non degenerate solution $u_\lambda\in {\mathcal C}_{\phi_1}^+(\Omega)$ to $(P_\lambda)$. Furthermore, the map
$$(-\lambda_0,\lambda_0)\ni\lambda\to u_\lambda\in {\mathcal C}_{\phi_1}^+(\Omega)$$
 is analytic and $\Vert u_\lambda-u_0\Vert_{{\mathcal C}_{\phi_1}(\Omega)}\to 0$ as $\lambda\to 0$  where $u_0$ is the unique solution to $(P)$. 
\end{proposition}
\proof
We  would like to apply the analytic version of the implicit function theorem. Given $u \in {\mathcal C}^+_{\phi_1 }$, we can check that the linearised operator $\partial_u F(\lambda,u): {\mathcal C}_{\phi_1 }(
\Omega ) \to {\mathcal C}_{\phi_1 }(
\Omega )$ is given by  
\begin{equation}\label{kab}
\partial_u F(\lambda,u )\phi=\phi+  (\Delta^2)^{-1}\Big( [\alpha K(x)u^{-\alpha-1} -\lambda f^{\prime}(u)]\phi \Big).
\end{equation}

 Note that $K(x)u^{-\alpha-1}\phi \in {\mathcal C}_{\phi_1 ^{-\alpha-\beta }}(\Omega )$. Indeed, for some $C>0$,
 $$
\| K(x)u^{-\alpha-1}\phi\|_{{\mathcal C}_{\phi_1 ^{-\alpha-\beta }}(\Omega )} \leq C \|\phi\|_{{\mathcal C}_{\phi_1 }(\Omega )}.
 $$
 Therefore from lemma \ref{jee}, we obtain a constant $\theta \in (0,1)$ (depending only on $\nu,\alpha$ and $\beta$) such that 
 $$ \Big\{ K(x)u^{-\alpha-1}\phi \Big\} \;\text{  bounded in }\;   {\mathcal C}_{\phi_1 ^{-\alpha-\beta }}(\Omega )  \Longrightarrow \Big\{(\Delta^2)^{-1}( K(x)u^{-\alpha-1}\phi)\Big\} \;\text{  bounded in }\;   C^{2,\theta}(\overline{\Omega}).$$
 We infer then that $\partial_u F(\lambda,u)$ is a compact perturbation of the identity and hence is a Fredholm operator of $0$-index. 
Next, we show that $\partial_u F(\lambda,u)$ is invertible for $\lambda \leq 0$.
If $\phi$ belongs to the kernel of $\partial_uF(\lambda,u)$, denoted by $N(\partial_uF(\lambda,u))$, we will have
\begin{eqnarray*}
\int_{\Omega}(\Delta\phi)^2{\rm d}x+ \alpha \int_{\Omega}K(x) u^{-\alpha-1} \phi^2 {\rm d}x - \lambda \int_{\Omega} f^{\prime}(u)\phi^2 {\rm d}x=0.
\end{eqnarray*}
Using $(f_2)$ and non positivity of $\lambda$ we  get $\phi\equiv\, 0$ from the above identity.  
Therefore, if $\lambda \leq 0$, we have $N(\partial_u F(\lambda,u))=\{0\}$ which implies that $\partial_u F(\lambda,u)$ is invertible. 

 Appealing to the real analytic version of  implicit function theorem (see \cite{BuTo}), we obtain a $\lambda_0>0$ and an analytic branch of solutions, $\lambda\to (\lambda ,u_\lambda)$ to $F(\lambda, u)=0$ for $\lambda\in (-\lambda_0,\lambda_0)$.  By taking $\lambda_0$ smaller if required, from the smoothness of the map $F$ we obtain that  $\partial_u F(\lambda,u_\lambda)$ is invertible for all $-\lambda_0<\lambda < \lambda_0$. That is, the solution $u_\lambda$ is non degenerate for all such $\lambda$. 
\qed\hfill
\begin{proposition}\label{nar}
There exists $\Lambda>0$ such that $(P_\lambda)$ admits no solution for $\lambda>\Lambda$.
\end{proposition}
\proof 
Using the assumption on $K$ and $(f_2)$ we note that for some positive constants $c_1,c_2$
$$
K(x)t^{-\alpha}+\lambda f(t) \geq c_1+ c_2 \lambda  t  \;\; \text{ for all }\;\; x\in \Omega, t >0.
$$
 Let $\lambda>0$. We multiply the equation in $(P_\lambda)$ by $\phi_1$ and use the above inequality to get for any solution $u$ to $(P_\lambda)$:
\begin{eqnarray*}
\lambda_1^2\int_{\Omega}u\phi_1{\rm d}x=\int_{\Omega}K(x)\phi_1u^{-\alpha}{\rm d}x+\lambda\int_{\Omega}f(u)\phi_1{\rm d}x\geq  c_1 \int_{\Omega} \phi_1{\rm d}x + c_2 \lambda\int_{\Omega}u\phi_1{\rm d}x.
\end{eqnarray*}
 Hence, necessarily,  $\lambda\leq \frac{\lambda_1^2}{c_2}$.   
\hfill\qed

We finally prove Theorem \ref{th1}.
\proof
Let ${\mathcal U}\eqdef \R\times {\mathcal C}_{\phi_1 }^+(\Omega )$. We first check that $F$ satisfies the conditions $(G1)-(G3)$ in order to apply Theorem \ref{theo9.1.1}. From the regularity estimate  in lemma \ref{jee},  we deduce that any bounded subset of ${\mathcal S}$ is relatively compact in $\R\times {\mathcal C}_{\phi_1}$, i.e. $(G1)$ holds. $(G2)-(G3)$ follow from the above proposition. Hence theorem \ref{theo9.1.1} asserts the existence of ${\mathcal A} \subset {\mathcal S}$ satisfying $(a)-(e)$. 

(i) follows from proposition \ref{A+}.\\
(ii) We first prove that assertion  in the alternative $(e)(i)$ of theorem \ref{theo9.1.1} occurs. We do this by ruling out the possibilities $(e)(ii)$ and $(e)(iii)$. 

The case $(e)(ii)$ can be ruled out as follows. Suppose there exists a sequence $\{(\lambda(s_n),u(s_n))\} \subset {\mathcal A}$ such that $(\lambda(s_n),u(s_n)) \to (\tilde\lambda, \tilde u) \in \partial {\mathcal U}$ in $\R\times {\mathcal C}_{\phi_1 }(\Omega )$ as $s_n\to\infty$. In particular, $\tilde u\not\in {\mathcal C}_{\phi_1}^+(\Omega)$. 
%From Lemma \ref{jee} and Proposition \ref{nar}, up to a subsequence (still denoted $(s_n)_{n\in\N}$), we infer that $\lambda(s_n)\to \tilde\lambda\in \R$ and $u(s_n)\to \tilde u$
% in ${\mathcal C}_{\phi_1 }(\Omega )$ as $n\to\infty$ with $\tilde u\not\in {\mathcal C}_{\phi_1 }^+(\Omega )$. 
Applying corollary \ref{dee}, we get  for some $C>0$ independent of $n$,
\begin{eqnarray}\label{hopf-n}
u(s_n)(x)\geq C\rho(x)\int_{\Omega}(K(y) u(s_n)^{-\alpha} +\lambda(s_n)f(u(s_n)) \rho(y) {\rm d}y \;\; \forall x \in \Omega.
\end{eqnarray}
Passing to the limit as $n\to\infty$ and using Fatou's Lemma, we get
\begin{eqnarray}\label{hopf-limit}
\tilde u(x)\geq C\rho(x)\int_{\Omega}(K(y) \tilde u^{-\alpha}+\tilde\lambda f(\tilde u ))\rho(y) {\rm d}y
\end{eqnarray}
which contradicts the assumption $\tilde u\not\in {\mathcal C}_{\phi_1}^+(\Omega)$ if $\tilde \lambda \geq 0$.

 Next we rule out alternative $(e)(iii)$. For that, we observe that $u_0$ is the unique solution to $(P_0)$ and from the implicit function theorem, ${\mathcal A}_0$ is the unique branch of solutions emanating from $(0,u_0)$. Therefore, ${\mathcal A}$ can not bend back to join the point $(0,u_0)$. 

Hence alternative $(e)(i)$ of theorem \ref{theo9.1.1} holds. From proposition \ref{nar}, the conclusion (ii) of theorem follows.

(iii) follows in view of (ii) and the fact that there is no solution for  all large $\lambda$ (prop. \ref{nar}).

(iv) and (v) of theorem \ref{th1}  follow directly from $(c)$ and $(d)$ of theorem \ref{theo9.1.1}.

(vi) We also note that (see Proposition \ref{A+}) since $\partial_u F(\lambda,u_\lambda)$ is an invertible operator for  $\lambda<0$, the negative portion of ${\mathcal A}$ i.e.,  ${\mathcal A} \cap (-\infty,0) \times {\mathcal C}_{\phi_1}(\Omega )$, is a single analytic curve (indeed a graph from the $\lambda$ axis) consisting of non-degenerate solutions $u_\lambda$. In particular,  this curve does not undergo any bifurcations.

 This completes the proof of the theorem.
\qed\hfill

\noindent {\bf Acknowledgements:} the authors were funded by IFCAM (Indo-French Centre for Applied Mathematics) UMI CNRS under the project "Singular phenomena in reaction diffusion equations and in conservation laws".
\end{document}